\documentclass[11pt]{article}
\usepackage{amssymb, amsmath, amsthm, epsf, epsfig, color}

\newlength{\originalbase}
\newcommand{\spacing}[1]{\setlength{\baselineskip}{#1\originalbase}}
\newcommand{\eps}{\varepsilon}

\newcommand{\h}{h}
\newcommand{\ul}{\underline}

\begin{document}
\spacing{1}
\newtheorem{theorem}{Theorem}[section]
\newtheorem{claim}{Claim}[theorem]
\newtheorem{prop}[theorem]{Proposition}
\newtheorem{remark}[theorem]{Remark}
\newtheorem{lemma}[theorem]{Lemma}
\newtheorem{corollary}[theorem]{Corollary}
\newtheorem{guess}[theorem]{Conjecture}
\newtheorem{conjecture}[theorem]{Conjecture}

\title{Sorting by Placement and Shift}

\author{Sergi Elizalde~\thanks{Department of Mathematics,
Dartmouth, Hanover NH 03755-3551, USA; sergi.elizalde@dartmouth.edu.}
\and
Peter Winkler~\thanks{Department of Mathematics, Dartmouth,
Hanover NH 03755-3551, USA; peter.winkler@dartmouth.edu.}}

\maketitle

\begin{abstract}
In sorting situations where the final destination of each item
is known, it is natural to repeatedly choose items and place
them where they belong, allowing the intervening items to
shift by one to make room.  (In fact, a special case of this
algorithm is commonly used to hand-sort files.)  However, it
is not obvious that this algorithm necessarily terminates.

We show that in fact the algorithm terminates after at most
$2^{n-1}-1$ steps in the worst case (confirming a conjecture of
L. Larson), and that there are super-exponentially many
permutations for which this exact bound can be achieved.
The proof involves a curious symmetrical binary representation.
\end{abstract}

\medskip

\section{The Problem}

Suppose that a permutation $\pi \in S_n$ is fixed and represented
by the sequence $\pi(1),\dots,\pi(n)$.  Any number $i$ with $\pi(i) \not= i$
may be ``placed'' in its proper position, with the numbers in positions
between $i$ and $\pi(i)$ shifted up or down as necessary.  Repeatedly
placing numbers until the identity permutation is achieved constitutes
a process we call {\em homing.}  One might imagine that the numbers are
written on billiard balls in a trough, as in Figure~\ref{fig:switch} below,
where the shift is a natural result of moving a ball to a new position.

\begin{figure}[htbp]
\epsfxsize220pt
$$\epsfbox{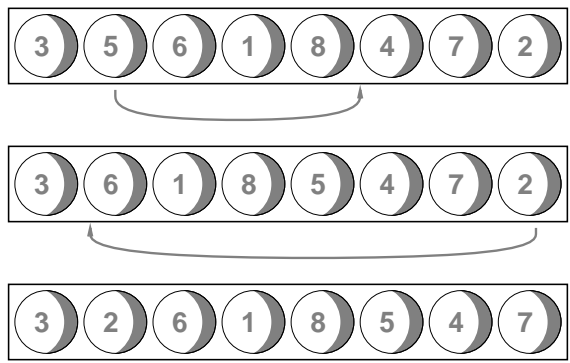}$$
\caption{Two steps in homing.}\label{fig:switch}
\end{figure}

To be precise, if $\pi(i)>i$ and $i$ is placed, the resulting permutation
$\pi'$ is given by
$$
\pi'(j)=\left\{ \begin{array}{ll}
	\pi(j) & \mbox{if $\pi(j) < i$ or $\pi(j) > \pi(i)$} \\
	i & \mbox{if $j = i$} \\
	\pi(j){-}1 & \mbox{if $i < \pi(j) < \pi(i)$}
	\end{array}
\right.
$$
and if $\pi(i)< i$, we have
$$
\pi'(j)=\left\{ \begin{array}{ll}
        \pi(j) & \mbox{if $\pi(j) < \pi(i)$ or $\pi(j) > i$} \\
        i & \mbox{if $j = i$} \\
        \pi(j){+}1 & \mbox{if $\pi(i) < j < i$}
	\end{array}
\right.
$$

The primary question we answer is: how many steps does homing
take in the worst case?

\section{History}

Despite its simplicity, homing seems not to have been considered before
in the literature; it arose recently as a result of a misunderstanding
(details below).  It is, of course, only in a loose sense a sorting algorithm
at all, since it requires that the final position of each item
be known, and presumes that it is desirable to sort ``in place.''
Thus, it makes sense primarily for physical objects.  Nonetheless, one
can imagine a situation where a huge linked list is to be sorted
in response to on-line information about where items ultimately belong;
then it may seem reasonable to place items as information is received,
allowing the items between to shift up or down by one.  {\em We do not
recommend this procedure!}

In hand-sorting files, it is common to find the alphabetically first
file and move it to the front, then find the alphabetically second
file and move it to the position behind the first file, et cetera.
This is a (fast) special case of homing.

Homing was brought to our attention by mathematician and reporter
Barry Cipra \cite{C}.  Cipra had been looking at John H. Conway's
``Topswops'' algorithm \cite{G}, in which only the leftmost number
is placed and intervening numbers are {\em reversed}.  Topswops terminates
when the 1 is in position, even if the rest of the numbers are still
scrambled.  Seeking to get everything in order, Cipra considered allowing
any not-at-home number to be placed, again reversing the intervening
numbers.  This algorithm does not necessarily terminate, however; a cycling
example ($7132568{\ul 4} \to {\ul 7}1348652 \to 5684317{\ul 2} \to
5271{\ul 3}486 \to 5231{\ul 7}486 \to 7132548{\ul 6} \to 71325684$)
was provided to Cipra by David Callan, of the University of Wisconsin \cite{Ca}.

When Cipra tried to explain his interest to his friend Loren Larson
(co-author of {\em The Inquisitive Problem Solver} \cite{VGL}) the latter
thought that the intervening items were to be shifted.  Cipra liked
the new procedure, especially as he was able to show it did always
terminate, and designed a game around it.  The game involved sorting
vertical strips of famous paintings (such as Picasso's {\em Guernica});
Cipra called it ``PermutARTions.'' A prototype of the game, renamed
``Picture This,'' has since been made by puzzle designer Oskar van Deventer.

The neatest proof known to us that homing always terminates is due to
Noam Elkies \cite{E}.  Since there are only finitely many states, non-termination
would imply the existence of a cycle; let $k$ be the largest number which
is placed {\em upward} in the cycle.  (If no number is placed upward, the
lowest number placed downward is used in a symmetric argument).  Once
$k$ is placed, it can be dislodged upward and placed again downward,
but nothing can ever push it below position $k$.  Hence it can never
again be placed upward, a contradiction.

\section{Outline}

In Section~\ref{sec:fast} we will consider fast homing, that is, the minimum
number of steps needed to sort from a given permutation $\pi$.  Among other
things we will see that homing can always be done in at most $n{-}1$ steps
(with a single worst-case example), and that there is an easy sequence of
choices which respects that bound.

In Section~\ref{sec:slow} we prove that homing cannot take more than
$2^{n-1}-1$ steps; in Section~\ref{sec:count}, we show that
there are super-exponentially many permutations which can support
exactly $2^{n-1}-1$ steps.

Finally, in Section~\ref{sec:final}, we wrap up and conclude with some
open questions.

\section{Fast Homing}\label{sec:fast}

A placement of either the least or the greatest number not currently in its home will be
called {\em extremal}; such a number $i$ will never subsequently be dislodged
from its home, since no other number will ever cross $i$ on its way.  Hence,

\begin{theorem}\label{thm:smallest} Any algorithm that always places the
smallest or largest available number will terminate in at most $n{-}1$ steps.
\end{theorem}

\begin{proof} After $n{-}1$ numbers are home, the $n$th must be as well.
\end{proof}

The algorithm which places the smallest not-at-home number is the one cited above,
often used to hand-sort files.  The precise number of steps required is the smallest
$j$ such that the files which belong in positions $j{+}1,j{+}2,\dots,n$
are already in the correct order.

Suppose placements are {\em random}, that is, at each step a uniformly random number
is chosen from among those that are not at home and then placed.  Let us say that
a permutation is in ``stage $k$'' if $k$ (but not $k{+}1$) of the extremal numbers are
home; thus, e.g., 1,2,3,7,4,6,5,8,9 is in stage 5, since 1, 2, 3, 8 and 9 are home.
There is no stage $n{-}1$. If $\pi$ is in stage $k$, for $k < n$, then with probability at least $2/(n{-}k)$,
the next placement will leave it in stage $k{+}1$ or higher.  It follows that
the expected number of placements needed to move up from stage $k$ is at most $(n{-}k)/2$,
and thus the total expected number of random placements needed to sort a permutation
cannot exceed $\sum_{k=0}^{n-2} \frac12 (n{-}k)$.  We conclude:

\begin{theorem} The expected number of steps required by random homing from $\pi \in S_n$
is at most $\frac14 (n(n{+}1)-2)$.
\end{theorem}

We now return our attention to well-chosen steps, seeking a lower bound.

\begin{theorem}\label{thm:subseq} Let $k$ be the length of the longest increasing
subsequence in $\pi$.  Then no sequence of fewer than $n{-}k$ placements can sort $\pi$.
\end{theorem}

\begin{proof}
Otherwise there are $k{+}1$ numbers which are never placed, and thus remain in their
original order; but that order cannot be correct, else it would constitute an
increasing subsequence of length $k{+}1$ in $\pi$.
\end{proof}

\begin{corollary} The reverse permutation $n,n{-}1,\dots,1$ requires $n{-}1$ steps.
\end{corollary}

Since \cite{LS,VK} the mean length of the longest increasing subsequence of a
{\em random} $\pi \in S_n$ is asymptotically $2\sqrt{n}$, we can deduce from
Theorem~\ref{thm:subseq} that a random permutation requires, on average, at
least $n-2\sqrt{n}$ steps to sort.

In general, $n$ minus the length of the longest increasing subsequence is not
enough steps to sort $\pi$.  An example (the only example for $n=5$) is provided
by the permutation 41352, which cannot be sorted using only two placements.

\begin{theorem} The reverse permutation is the only case requiring $n{-}1$ steps.
\end{theorem}

\begin{proof} By induction on $n$, the $n=1$ case being trivial.  If $\pi$ is not the
reverse permutation, there must be $i<j$ with $\pi^{-1}(i) < \pi^{-1}(j)$.  Moreover,
for $n>2$ it cannot be that the only such pair is $i=1$, $j=n$.  Hence either 1 or
$n$ can be placed still leaving a non-reverse permutation of the remaining numbers,
which can be sorted in $n{-}2$ steps by the induction hypothesis.
\end{proof}

Existence of a unique worst case (especially this one) for a sorting algorithm is
hardly surprising.  When we instead maximize the number of steps, something
startlingly different takes place.

\section{Slow Homing}\label{sec:slow}

How long can homing take?  It is not hard to verify that if one begins with the
permutation $2,3,\dots,n{-}1,n,1$ and always places the left-most not-at-home entry,
the result is $2^{n-1}-1$ steps before the identity permutation is reached (via
the familiar ``tower of Hanoi'' pattern).  Larson \cite{L} conjectured that
$2^{n-1}-1$ is the maximum.  Indeed, although many other, more complex, permutations
can also support $2^{n-1}-1$ steps, none permit more.

\begin{theorem}\label{thm:main} Homing always terminates in at most $2^{n-1}-1$ steps.
\end{theorem}

To prove Theorem~\ref{thm:main} we will require several lemmas and some
backward analysis.  The reverse of homing, which we will call {\em evicting},
entails choosing a number which {\em is} where it belongs and {\em dis}placing it,
that is, putting it somewhere else, again shifting the intervening values up or
down by one.  Our objective is then to show that beginning with the identity permutation
on $\{1,2,\dots,n\}$, at most $2^{n-1}-1$ displacements are possible.  This
is trivial for $n=1$ and we will proceed by induction on $n$.

\begin{lemma}\label{lemma:stage1} After $2^{n-2}$ displacements, both 1 and $n$
have been displaced and will never be displaced again.
\end{lemma}

\begin{proof}
Let us observe first that the numbers 1 and $n$ can each be displaced only
once, since neither can subsequently be shifted back to its proper end.
(Equivalently, in the forward direction, each can be placed only once.)

If after $2^{n-2}$ displacements one of these values (say, the number 1) has
never been displaced, then it remains where it began and played no role
whatever in the process.  Hence the remaining $n{-}1$ numbers allowed
more than $2^{(n-1)-1}-1$ displacements, contradicting the induction hypothesis.
\end{proof}

We now endeavor to show that at most $2^{n-2}-1$ displacements can take place in
the second stage, after 1 and $n$ have been displaced.  To do this we associate
with each intermediate state $\pi$ a {\em code} $\alpha(\pi)$, and with each code $\alpha$,
a weight $w = w(\alpha)$.

The code is a sequence $\alpha = (a_2,a_3,\dots,a_{n-1})$ of length $|\alpha| = n{-}2$
from the alphabet $\{+,-,0\}$.  Given a permutation $\pi$, recall from above
that $\pi(i)$ represents the value in the $i$th position from the left, and
therefore $\pi^{-1}(i)$ represents the position of the number $i$.
Define $\alpha(\pi)$ by putting $a_i = +$ if $\pi^{-1}(i) >i$, that is, if the
number $i$ is to the right of where it belongs.  Similarly, $a_i = -$
if $\pi^{-1}(i) <i$, and $a_i = 0$ if $\pi^{-1}(i) = i$.  Thus, a
number $i$ can be displaced if and only if $a_i = 0$.

Figure~\ref{fig:code} shows an example of a permutation and its code.

\begin{figure}[htbp]
\epsfxsize220pt
$$\epsfbox{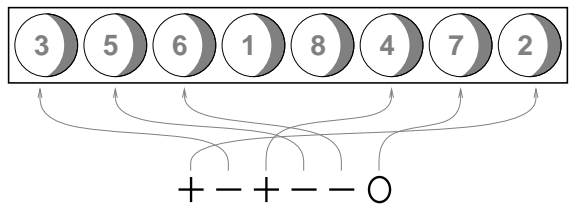}$$
\caption{A permutation and its code.}\label{fig:code}
\end{figure}

The weight $w(\alpha)$ is defined for codes of all lengths by recursion.
If $a_i = 0$ for each $i$, we put $w(\alpha) = 0$.

For each $i$ such that $a_i = -$, let $d_i=i{-}2$; for each $i$ with
$a_i=+$, let $d_i=n{-}1{-}i$.  Thus, $d_i$ represents the number of symbols
to the left of a $-$ or to the right of a $+$.

Let $i$ be the index maximizing $d_i$; if there are two such values (necessarily
one representing a $-$ and the other a $+$), let $i$ be the one for which $a_i=-$.
(We will see that this choice has no effect.)
Let $\alpha[i]$ the code of length $|\alpha|{-}1$ obtained by deleting the $i$th entry of $\alpha$.
Then $w(\alpha) = w(\alpha[i]) + 2^{d_i}$.

If the code consists only of 0's and $+$'s, then changing the $+$'s to 1's
gives the binary representation of $w(\alpha)$.  If instead there are no
$+$'s in the code, then changing every $-$ to a 1 gives the {\em reverse}
binary representation of $w(\alpha)$.  Thus the code is a sort of double-ended
binary representation of $w(\alpha)$.  Figure~\ref{fig:weight} shows the recursive derivation
of $w(\alpha)$ from a sample code $c$; the gray arrows point to the entry next to be stripped.

\begin{figure}[htbp]
\epsfxsize220pt
$$\epsfbox{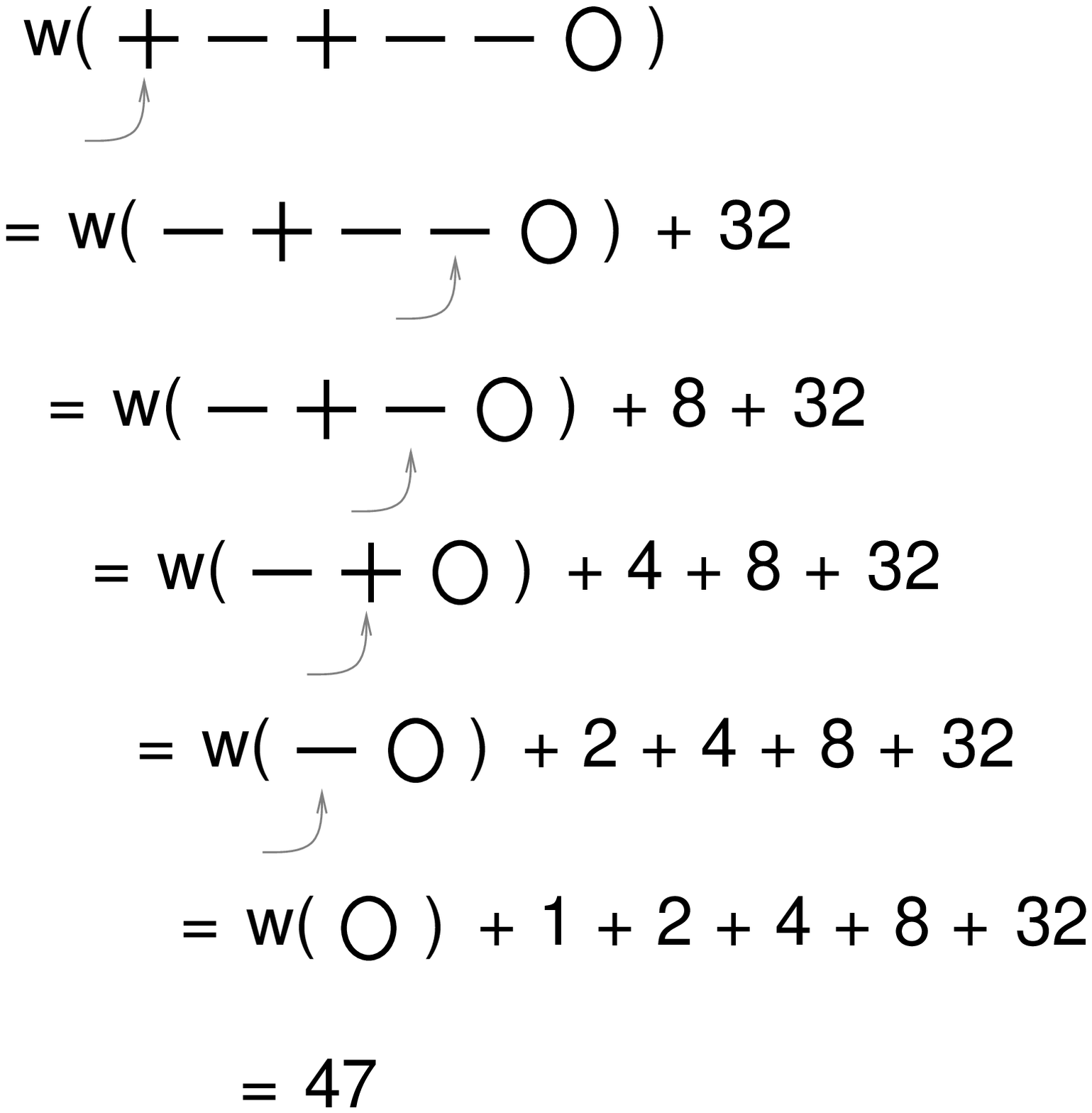}$$
\caption{Derivation of the weight of a code.}\label{fig:weight}
\end{figure}

We next make some elementary observations about codes and their weights.

\begin{lemma}\label{lemma:extremes}
The minimum of $w(\alpha)$ over codes $\alpha$ of length $k$ is 0, for the
all-0 code, and the maximum is $2^k-1$, for codes of the form $+^p -^q$.
\end{lemma}
\begin{proof}
This follows from the fact that during the recursion, in reducing the length
of $\alpha$ from $k$ to $k{-}1$ the weight change is at most $2^{k-1}$, and
achieves that value only when a $-$ is deleted from the right or a $+$ from the left.
\end{proof}

\begin{lemma}\label{lemma:block}
Let $\alpha = \beta +^p \gamma -^q \delta$ where $|\beta|=|\delta|$, $\beta$ contains no $+$,
$\delta$ contains no $-$, and $\gamma$ neither begins with $+$ nor ends with $-$.
Then $w(\alpha) = w(\beta\gamma\delta) + 2^{p+|\gamma|+q+|\beta|} - 2^{|\gamma|+|\beta|}$.
\end{lemma}

\begin{proof}
Immediate from the definition of $w(\alpha)$, since the indicated blocks of $+$'s and $-$'s will
be eliminated before any other entries.
\end{proof}

\begin{corollary}\label{cor:tiebreak} The definition of the weight of a code
does not depend on how ties are broken when $d_i=d_j$.
\end{corollary}

\begin{proof} If in the code $\alpha$ $d_i = d_j$, where $a_i = +$ and $a_j = -$,
and $i<j$, then the situation is as in Lemma~\ref{lemma:block} and, irrespective of
the tiebreak mechanism, the entries of the blocks will be taken next and the
resulting weight is the same.  If $i>j$ (thus all $+$'s in $\alpha$ lie to the
right of all $-$'s), the removal of $a_i$ has no effect on $d_j$ and vice-versa,
so the two operations trivially commute.
\end{proof}

\begin{lemma}\label{lemma:zero}
For any codes $\gamma$ and $\delta$, where $\gamma$ has no $+$, $w(\gamma \delta 0)
\le w(\gamma\delta) + 2^{|\delta|}-1$.
\end{lemma}

\begin{proof}
Clearly the presence of an extra 0 at the end increments $d_i$ by 1 whenever
$a_i=+$, so if a $+$ is stripped from both $\gamma\delta 0$ and from $\gamma\delta$
the weight change is doubled for the former.  In the extreme case, if $\delta = +^{|\delta|}$,
the difference $w(\gamma \delta 0) - w(\gamma\delta)$ is therefore exactly
$\sum_{j=0}^{|\delta|-1} = 2^{|\delta|}-1$.

When a $-$ is stripped, the weight change is the same, so it would appear that
$w(\gamma \delta 0) - w(\gamma\delta)$ would then be smaller.  The difficulty is that
the incremented weights may cause symbols to be stripped in a different order in the two codes.

To fix that problem we employ Corollary~\ref{cor:tiebreak}.  In deriving
$w(\gamma \delta)$ ties are broken in favor of $-$ (as in the definition of $w$);
when deriving $w(\gamma \delta 0)$, in favor of $+$.  This will result in symbols being
stripped in the same order from the two codes, up to the point where all $+$'s
lie to the right of all $-$'s.  After that, $d_i$ for a given symbol is unaffected by
stripping symbols of the opposite sign, so the order becomes immaterial.
\end{proof}

\begin{lemma}\label{lemma:change}
Let $\alpha$ be any code, and $\beta = (b_2,\dots,b_{n-1})$ the result of
changing some $a_i=0$ to $b_i=+$ or $b_i=-$.  Then $w(\beta)>w(\alpha)$.
\end{lemma}

\begin{proof}
Suppose $b_i=-$; the other case is symmetric (and uses the reflected form of Lemma~\ref{lemma:zero}).

The derivations of $w(\beta)$ and $w(\alpha)$ are the same until
$b_i$ is stripped. Let $\beta'$ and $\alpha'$ be the corresponding codes at that point, right before $b_i$
is stripped.
We can write $\beta'= \gamma \delta b_i \eps$, where $\gamma$ contains no $+$,
$\eps$ no $-$, and $|\eps|\le|\gamma|$.

We then have
$$
w(\beta') = w(\gamma \delta \eps) + 2^{|\gamma|+|\delta|} =
w(\gamma \delta 0^{|\eps|}) + w(\eps) + 2^{|\gamma|+|\delta|}
\ge w(\gamma \delta 0^{|\eps|+1}) - (2^{|\delta|+|\eps|}-1) + w(\eps) + 2^{|\gamma|+|\delta|}
$$
(by Lemma~\ref{lemma:zero})
$$
> w(\gamma \delta 0^{|\eps|+1}) + w(\eps) = w(\gamma \delta 0 \eps) = w(\alpha').
$$
\end{proof}

\begin{lemma}\label{lemma:weight-up} Let $\pi$ be any permutation of $\{1,\dots,n\}$ in which
$\pi(1) \not= 1$ and $\pi(n) \not= n$, and let $\pi'$ be the result of
applying some displacement to $\pi$.  Let $\alpha = \alpha(\pi)$ and $\alpha' = \alpha(\pi')$;
then $w(\alpha') > w(\alpha)$.
\end{lemma}

\begin{proof} A displacement chases a value $i$ away from home, thus causing the
0 in position $i$ of the code $\alpha$ to become a $+$ or a $-$. Assume the latter
(the alternative argument is symmetric).  Since the number $i$ is being moved to
the left, other numbers will move right one position or stay where they are;
thus, the other entries of $\alpha$ can change only from $-$ to 0 or from $0$ to $+$.
We care only about the former possibility, since by Lemma~\ref{lemma:change},
changing a $0$ to a $+$ can only increase $w(\alpha)$.

However, any change of a $-$ to a $0$ in $\alpha$ must have taken place to the left
of the entry $a_i$, because a number bigger than $i$ but to its left in $\pi$
cannot get to its home (to the right of position $i$) when $i$ is displaced.
Again by Lemma~\ref{lemma:change}, we can assume that all the $-$'s to the left of $a_i$
change to $0$. Let $j$ be the position of the rightmost $-$ to the left of $a_i$ in $\alpha$
(if there is no such $-$, let $j=1$). Let $2^k$ be the contribution of the $-$ in position $i$ in the computation of $w(\alpha')$.

If there are any $+$ entries between $a_j$ and $a_i$ that are stripped after the $-$ in position $i$ in $\alpha'$, then
their contribution to $w(\alpha')$ is less (by a factor of 2) than their contribution to $w(\alpha)$.
Let $\ell$ be the number of such $+$'s. The total contribution of these $+$'s to $w(\alpha')$ is at most
$2^{k-1}+2^{k-2}+\dots+2^{k-\ell}=2^{k-\ell}(2^\ell-1)$. Thus, the difference between their contribution to
$w(\alpha)$ and their contribution to $w(\alpha')$ is at most $2^{k-\ell}(2^\ell-1)$ as well.
On the other hand, the total contribution to $w(\alpha)$ of the $-$ entries to the left of $a_i$ in $\alpha$ is at most $2^{k-\ell}-1$,
since each adds a different power of 2 less than $2^{k-\ell}$. We conclude that
$$w(\alpha')\ge w(\alpha)+2^{k}-(2^{k-\ell}-1)-2^{k-\ell}(2^\ell-1)>w(\alpha).$$
\end{proof}

Theorem~\ref{thm:main} is an easy consequence of Lemmas~\ref{lemma:stage1}~and~\ref{lemma:weight-up}.
In fact nothing prevents us from associating to each $\pi \in S_n$ a code of
full length $n$, and applying the above argument to conclude directly that
there can be no more than $2^n-1$ displacements.  However, this falls short of
the desired result by a factor of 2 (as does an argument based on Elkies' finiteness
proof); hence the 2-stage argument above.

\section{Counting Bad Cases}\label{sec:count}

The proof of Theorem~\ref{thm:main} tells us somewhat more about the
worst-case structure of eviction, that is, about the digraph on $S_n$ which
boasts an arc from $\pi$ to $\pi'$ when $\pi'$ is among the longest-lived
permutations that can be reached from $\pi$ by a single displacement.
We are particularly interested in the set $M_n$ of permutations at maximum distance
$2^{n-1}-1$ from the source (the identity permutation), since these are
the worst-case starting points for homing.  The proof shows that each permutation
in $M_n$ must have a code of the form $+^k -^{n-2-k}$, but the converse does not hold
in general.

Let the {\em height} $\h(\pi)$ of a permutation $\pi$ be the distance to $\pi$
from the source in the above digraph (equivalently,
the maximum length of a sequence of placements from $\pi$ to the identity).
In the rest of the paper, let $\tau_n$ denote the permutation $n,2,3,\dots,n{-}1,1$.

\begin{lemma}\label{lem:n_1}
$\h(\tau_n)=2^{n-2}$.
\end{lemma}

\begin{proof}
The only entries that can be placed in the first step are $n$ and $1$.
By symmetry, we can assume that $n$ is placed first. The steps after that are
equivalent to homing the permutation $2,3,\dots,n{-}1,1$ of length $n{-}1$.
By Theorem~\ref{thm:main} and the observation above it, we know that
$\h(2,3,\dots,n{-}1,1)=2^{n-2}-1$.
\end{proof}

\begin{lemma}\label{lem:carry}
For any permutation with code $\alpha=+^i 0^k -^j$, there is a sequence of
$2^{k-1}$ displacements that ends in a permutation with code $+^i 0^{k-1} -^{j+1}$.
Moreover, all the displacements in the sequence are unique, except for possibly the last one.
\end{lemma}

\begin{proof}
To show existence, we will prove that if
$$
\pi=\pi(1),\dots,\pi(i{+}1),\ul{i{+}2,\dots,i{+}k{+}1},\pi(i{+}k{+}2),\dots,\pi(n)
$$
is a permutation with code $\alpha$ (fixed points have been underlined), we can perform $2^{k-1}$ displacements
and end with the permutation that is obtained from $\pi$ by transposing the
entries $\pi(i{+}1)$ and $i{+}k{+}1$. We proceed by induction on $k$.
The result is trivial for $k=1$. Assume that $k\ge2$.
By the induction hypothesis, we can perform $2^{k-2}$ displacements on
$\pi$ to transpose $\pi(i{+}1)$ with $i{+}k$. Let $\pi'$ be the resulting permutation.
Again via induction, by performing $2^{k-3}$ displacements on $\pi'$ we can transpose
$\pi'(i{+}1)=i{+}k$ and $i{+}k{-}1$.  If we repeat this process, after $2^{k-2}+2^{k-3}+\dots+1$
displacements we obtain the permutation
$$
\pi(1),\dots,\pi(i),i{+}2,i{+}3,\dots,i+k,\pi(i{+}1),\ul{i{+}k{+}1},\pi(i{+}k{+}2),\dots,\pi(n).
$$
Finally, displacing $i{+}k{+}1$ to position $i{+}1$, we obtain
$$
\pi(1),\dots,\pi(i),i{+}k{+}1,\ul{i{+}2,i{+}3,\dots,i{+}k},\pi(i{+}1),\pi(i{+}k{+}2),\dots,\pi(n)$$
as desired. The code of this permutation is $+^i 0^{k-1} -^{j+1}$.

To see uniqueness, note first that the $+$'s and $-$'s in $\alpha=+^i 0^k -^j$ can never
be changed by displacements. Since $w(+^i 0^{k-1} -^{j+1})=w(+^i 0^k -^j)+2^{k-1}$,
we have by Lemma~\ref{lemma:weight-up} that the only way that the code
can evolve from $+^i 0^k -^j$ to $+^i 0^{k-1} -^{j+1}$ in $2^{k-1}$ steps is if the
weight increases by one at each step. This means that at each step, the leftmost
$0$ in the code is changed to a $-$ and the $-$'s to its left are changed back to $0$'s.
For a $-$ to become a $0$ in a displacement step, it must correspond to an entry
$r$ in position $\pi^{-1}(r)=r{-}1$. But this condition can only hold if the sequence
of displacements (except possibly the last one) is the one described above.
Note that the last displacement (of $i{+}k{+}1$) can be done into any position
$\le i{+}1$, so there are $i{+}1$ choices for it.
\end{proof}

We will refer to the sequence of $2^{k-1}$ displacements described in the proof of
Lemma~\ref{lem:carry} as {\em firing $i{+}k{+}1$ to the left}. If the last displacement
(of $i{+}k{+}1$) is done into position $i{+}1$, we call it a {\em short firing of $i{+}k{+}1$ to the left}.
In a symmetric fashion, we can define a {\em firing of $i{+}2$ to the right}, which is called a
{\em short firing} if the last displacement (of $i{+}2$) is into position $i{+}k{+}2$. In this case,
the code changes from $+^i 0^k -^j$ to $+^{i+1} 0^{k-1} -^j$.

\begin{lemma}\label{lem:charMn}
A permutation belongs to $M_n$ if and only if it can be obtained from $\tau_n$
by successively applying $n{-}2$ left and right firings.
\end{lemma}

\begin{proof} See Appendix.
\end{proof}

\begin{corollary}
For $n\ge 2$, $|M_n|\le (n-1)!$.
\end{corollary}

\begin{proof}
It follows from Lemma~\ref{lem:charMn}, together with the fact that a left (resp.\ right)
firing on a permutation with code $+^i 0^k -^j$ can be done in $i{+}1$ (resp.\ $j{+}1$) ways.
\end{proof}

Note that a given permutation in $M_n$ may be obtained through different sequences of firings,
so the actual size of $M_n$ will be less than $(n{-}2)!$ in general. We first give an easy lower bound on $|M_n|$.

\begin{prop}
$|M_n|\ge 2^{n-2}$.
\end{prop}

\begin{proof} See Appendix.
\end{proof}

This bound can be improved if we allow any firings to the left but only short firings to the right.
Denote by $B_n$ the $n$th Bell number, which gives the number of partitions of the set $\{1,2,\dots,n\}$.
The asymptotic growth of the Bell numbers is super-exponential. More precisely,
$$B_n\sim\frac{1}{\sqrt{n}}\,\lambda(n)^{n+1/2}e^{\lambda(n)-n-1},$$
where $\lambda(n)=\frac{n}{W(n)}$, and $W$ is the Lambert $W$-function, defined by $W(n)e^{W(n)}=n$.

\begin{theorem}
$|M_n|\ge B_{n-1}$.
\end{theorem}

\begin{proof} See Appendix.
\end{proof}

It follows from Lemma~\ref{lem:charMn} that if we allow arbitrary right firings and we record a right
firing of $i{+}2$ into position $s\ge i{+}k{+}2$ by $R_{s-i-k-2}$,
then every sequence of left and right firings applied to $\tau_n$ can be encoded
as a word of length $n{-}2$ on the alphabet $\{R_0,R_1,\dots,L_0,L_1,\dots\}$
with the restriction that every occurrence of  $L_s$ (resp. $R_s$) must have at least $s$ $R_\star$'s
(resp. $L_\star$'s) to its left, for every $s$.
As mentioned above, different such words can produce the same permutation,
due to the fact that sometimes left and right firings commute.
More precisely, we have the relations $L_{t-1}R_s=R_{s-1}L_t$ for every $s,t\ge 1$.
These relations partition the set of words into equivalence classes, one for each permutation in $M_n$.
We can select a canonical representative for each class if we replace each occurrence of $L_{t-1}R_s$
(with $s,t\ge 1$) with $R_{s-1}L_t$, until there are no more occurrences left. For example,
the representative for the class containing $L_0R_1R_0L_1R_2R_1$ is $R_0L_1R_0R_1R_0L_3$,
and the corresponding permutation is $7,6,8,1,3,2,5,4$.
It is not hard to see that regardless of the order in which these replacements are made,
we end with a unique word where no $R_s$ with $s\ge1$ is preceded by an $L_\star$.
We have proved the following result.

\begin{prop}\label{prop:words}
There is a bijection between $M_n$ and the set $W_n$ of words of length $n{-}2$ over the alphabet $\{R_0,R_1,\dots,L_0,L_1,\dots\}$ satisfying:
\begin{enumerate}
\item every occurrence of $L_s$ has at least $s$ $R_\star$'s to its left, for every $s$,
\item every occurrence of $R_s$ has at least $s$ $L_\star$'s to its left, for every $s$, and
\item no $R_s$ with $s\ge1$ is immediately preceded by an $L_\star$.
\end{enumerate}
\end{prop}

For $i,j\ge1$, let $f_{i,j}=|\{\pi\in M_{i+j}:\alpha(\pi)=+^{i-1}-^{j-1}\}|$.
Let $F(u,v)=\sum_{i,j\ge1} f_{i,j} \, u^iv^j$.

\begin{theorem} The generating function $F(u,v)$ satisfies the following partial differential equation:
$$F(u,v)=uv+uv\frac{\partial}{\partial u}F(u,v)+uv\frac{\partial}{\partial v}F(u,v)-u^2v^2\frac{\partial^2}{\partial u\partial v}F(u,v).$$
\end{theorem}

\begin{proof}
By Proposition~\ref{prop:words}, $f_{i,j}$ is the number of words in $W_{i+j}$ with $i{-}1$ $R_\star$'s and $j{-}1$ $L_\star$'s. Let $X_{i,j}\subseteq W_{i+j}$ denote this set.
We will show that, for $i,j\ge1$ with $i+j\ge3$, the numbers $f_{i,j}$ satisfy the recurrence
\begin{equation}\label{eq:recurrence} f_{i,j}=i f_{i,j-1}+j f_{i-1,j}-(i-1)(j-1)f_{i-1,j-1},\end{equation}
where we define $f_{i,j}=0$ whenever $i=0$ or $j=0$. The other initial condition is $f_{1,1}=1$, which corresponds to the empty word.


Let $i,j\ge1$ with $i+j\ge2$, and let $w\in X_{i,j}$. If the last letter of $w$ is an $L_\star$, then deleting it we get a word in $X_{i,j-1}$. Conversely, given a word in $X_{i,j-1}$, we can append to
it any letter $L_s$ with $0\le s\le i-1$ to obtain a word in $X_{i,j}$. This explains the term $i f_{i,j-1}$ in (\ref{eq:recurrence}).

If the last letter of $w$ is an $R_\star$, then deleting it we get a word in $X_{i-1,j}$. For the converse we have to be a little more careful. Given a word in $X_{i-1,j}$, we can append to
it any letter $R_s$ with $0\le s\le j-1$ to obtain a word in $X_{i,j}$ (this gives the term $j f_{i-1,j}$), except if the word in $X_{i-1,j}$ ends with an $L_\star$.
In this case, we are not allowed to append any $R_s$ with $1\le s\le j-1$, since that
would violate the third condition in the definition of $W_n$. The count the number of forbidden situations, observe that there are  $(i-1) f_{i-1,j-1}$ words in $X_{i-1,j}$ ending
with an $L_\star$, and to each one of them we could append an $R_s$ with $1\le s\le j-1$. This is why we subtract $(i-1)(j-1)f_{i-1,j-1}$ in~(\ref{eq:recurrence}).

Using equation~(\ref{eq:recurrence}) for the coefficients of $F(u,v)$ we get
\begin{multline*}
F(u,v)=uv+\sum_{i+j\ge 3} f_{i,j} \, u^iv^j\\
= uv+ uv \sum_{i+j\ge 3} i f_{i,j-1}\,u^{i-1} v^{j-1}+ uv \sum_{i+j\ge 3} j f_{i-1,j}\,u^{i-1}v^{j-1}- u^2v^2 \sum_{i+j\ge 3} (i-1)(j-1)f_{i-1,j-1}\,u^{i-2}v^{j-2}\\
= uv+uv \sum_{i,j\ge 1} i f_{i,j}\,u^{i-1} v^{j}+ uv \sum_{i,j\ge 1} j f_{i,j}\,u^{i}v^{j-1}- u^2v^2 \sum_{i,j\ge 1} i j f_{i,j}\,u^{i-1}v^{j-1}\\
= uv+uv \frac{\partial}{\partial u}F(u,v)+ uv \frac{\partial}{\partial v}F(u,v) - u^2v^2 \frac{\partial^2}{\partial u\partial v}F(u,v).
\end{multline*}
\end{proof}

Note that $|M_n|$ is the coefficient of $t^n$ in $F(t,t)$. The first few values of this sequence are $1,2,5,16,62,280,1440,8296,52864,\dots$.

\section{Conclusions}\label{sec:final}

The ``homing'' sort proposed by Cipra and Larson is a natural way to put
a permutation in order, and does work---eventually.  While $n{-}1$ well-chosen steps
will always succeed (with only one worst-case permutation), poorly-chosen steps
lead---for super-exponentially many permutations of $\{1,\dots,n\}$---to the precise
maximum number of steps, namely $2^{n-1}-1$.

The asymptotic behavior of the number of worst-case permutations seems to be strictly between factorial growth and the growth of Bell numbers.
If Figure~\ref{fig:growth} we have plotted the graphs of $(n-1)!^{1/n}$, $|M_n|^{1/n}$, and $B_{n-1}^{1/n}$, for $n\le 80$. It is known
that $(n-1)!^{1/n}\sim \frac{n}{e}$ and $B_{n-1}^{1/n}\sim \frac{n}{e\, W(n)}$, where $W$ is the $W$-Lambert function. Thomas Prellberg~\cite{Pr}
conjectures that $|M_n|^{1/n}\sim \frac{n}{2e}$. He argues that when $i\approx j$, (\ref{eq:recurrence}) suggests that $f_{i,j}$ can be approximated by $g_{i+j}$, where $g_n$ satisfies
the recurrence
$$g_{n+1}=n\,g_n-\frac{n^2}{4}\,g_{n-1},$$ from where the asymptotic behavior follows.

\begin{figure}[htbp]
$$\epsfig{file=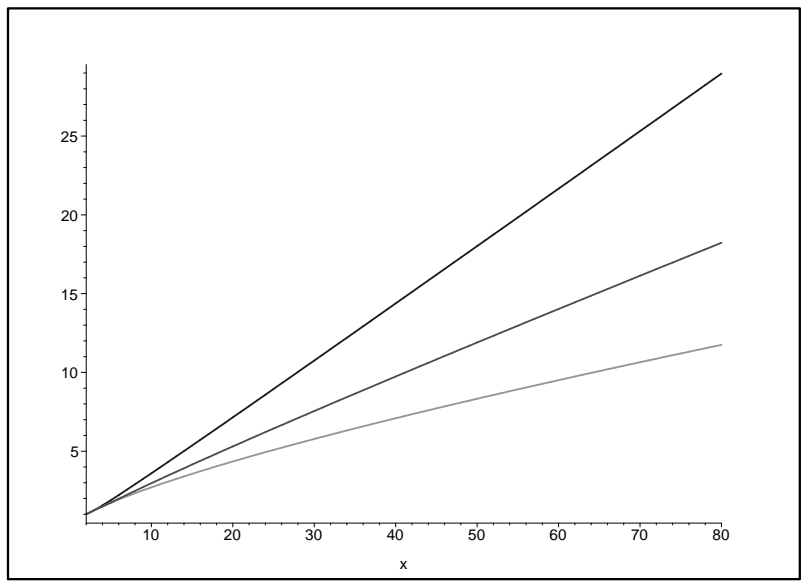,height=300pt,angle=-90}$$
\caption{The graphs of $(n-1)!^{1/n}$, $|M_n|^{1/n}$, and $B_{n-1}^{1/n}$, from top to bottom.}\label{fig:growth}
\end{figure}

We leave the calculation of the exact number of worst-case permutations, and
the precise behavior of homing (optimal, pessimal or random) on random permutations, to others.

\pagebreak

\section{Appendix}

Here we provide the (relatively straightforward) proofs of Lemma 6.3, Proposition 6.5 and Theorem 6.6.

\bigskip

\noindent
{\bf Lemma 6.3}. {\it A permutation belongs to $M_n$ if and only if it can be obtained from $\tau_n$
by successively applying $n{-}2$ left and right firings.}

\begin{proof}
Let us first show sufficiency. Note that $\alpha(\tau_n)=0^{n-2}$. The first
firing transforms this code into $0^{n-3}-$ or $+0^{n-3}$ using $2^{n-3}$ displacements.
The second firing uses $2^{n-4}$ displacements, and so on. After $2^{n-3}+2^{n-4}+\dots+1=2^{n-2}-1$ displacements,
we end with a permutation $\sigma$ with code $+^k -^{n-2-k}$ for some $k$. By Lemma~\ref{lem:n_1},
$\h(\sigma)\ge 2^{n-2}+2^{n-2}-1=2^{n-1}-1$, and by Theorem~\ref{thm:main} this is an equality, so $\sigma\in M_n$.

Conversely, by Lemmas~\ref{lemma:stage1}~and~\ref{lemma:weight-up}, any permutation of height
$2^{n-1}-1$ has to be obtained from $\tau_n$ by performing $2^{n-2}-1$ displacements,
each one increasing the weight by one.  If the first displacement on $\tau_n$ introduces
a $-$ to the code, then the first $2^{n-3}$ displacements must constitute a left firing.
Otherwise, either one of these displacements would increase the weight by more than one,
or a $+$ would be introduced before the code is $0^{n-3}-$, which would cause the weight
to increase by more than one at a later displacement. Therefore, the first $2^{n-3}$
displacements on $\tau_n$ must constitute a right or a left firing.
Repeating this argument, the same is true for the the next $2^{n-4}$ displacements, and so on.
\end{proof}

\bigskip

{\noindent}
{\bf Proposition 6.5}. $|M_n|\ge 2^{n-2}$.

\begin{proof}
We show that if we start from $\tau_n$ and perform only short firings, then no permutation is obtained in more than one way.
To see this, consider
$$
\pi=\pi(1),\dots,\pi(i{+}1),\ul{i{+}2,\dots,i{+}k{+}1},\pi(i{+}k{+}2),\dots,\pi(n)
$$
with $\alpha(\pi)=+^i 0^k -^{n{-}2{-}i{-}k}$. If we perform a short firing of $i{+}k{+}1$ to the left, we obtain
$$
\pi(1),\dots,\pi(i),i{+}k{+}1,\ul{i{+}2,\dots,i{+}k},\pi(i{+}1),\pi(i{+}k{+}2),\dots,\pi(n).
$$
Regardless of what short firings we perform after this,
$i{+}k{+}1$ will always remain to the left of $\pi(i{+}1)$, and $i{+}2$ will always remain to the left of $\pi(i{+}k{+}2)$.
However, if we had instead performed on $\pi$ a short firing of $i{+}2$ to the right,
then we would have obtained
$$
\pi(1),\dots,\pi(i{+}1),\pi(i{+}k{+}2),\ul{i{+}3,\dots,i{+}k{+}1},i{+}2,\pi(i{+}k{+}3),\dots,\pi(n),
$$
and any subsequent short firings on this permutation would preserve the relative position
of $\pi(i{+}1)$ to the left of $i{+}k{+}1$, and $\pi(i{+}k{+}2)$ to the left of $i{+}2$.
It follows that each of the $2^{n-2}$ possible sequences of short left and right firings
that can be applied to $\tau_n$ results in a different permutation.
\end{proof}

\bigskip

\noindent
{\bf Theorem 6.6}. $|M_n|\ge B_{n-1}$.

\begin{proof}
Let $P_n\subseteq M_n$ be the set of permutations that can be obtained from $\tau_n$ by performing
a sequence of $n{-}2$ arbitrary firings to the left and short firings to the right.
For a permutation $\pi$ as in the above proof, a firing of $i{+}k{+}1$ to the left results in
$$
\pi(1),\dots,\pi(s{-}1),i{+}k{+}1,\pi(s),\dots,\pi(i),\ul{i{+}2,\dots,i{+}k},\pi(i{+}1),\pi(i{+}k{+}2),\dots,\pi(n),
$$
for some $1\le s\le i{+}1$. After this firing, any sequence of firings to the left and short firings to the right
will leave $i{+}2$ to the left of $\pi(i{+}k{+}2)$, and will preserve the fact that $i{+}k{+}1$
lies to the right of $\pi(s{-}1)$ (if $s\ge2$) and to the left of $\pi(s)$.
On the other hand, if we had instead performed on $\pi$ a short firing of $i{+}2$ to the right,
then any further firings would leave $\pi(i{+}k{+}2)$ to the left of $i{+}2$. Thus,
every permutation $\sigma\in P_n$ uniquely determines the sequence of left and short right firings
that have to be applied to $\tau_n$ in order to obtain $\sigma$.

Next we will determine how many such sequences there are. After performing $j$ left firings and $i$ right firings on $\tau_n$, the code of the resulting
permutation is $+^i 0^k -^j$. If we now fire $i{+}k{+}1$ to the left, we have $i{+}1$ choices for the position $s$ to which
the entry $i{+}k{+}1$ is displaced.  Recording a left firing into position $s$ by $L_{i+1-s}$ and
a short right firing by $R$, each permutation in $P_n$ can be encoded uniquely as a word of length $n{-}2$ on the alphabet
$\{R,L_0,L_1,\dots\}$ with the restriction that every occurrence of $L_s$ must have at least $s$ $R$'s preceding it, for every $s$.

We claim that the number of such words containing $k$ $R$'s equals the number of partitions of $\{1,2,\dots,n{-}1\}$ with $k{+}1$ blocks,
for every $0\le k\le n{-}2$. Here is a bijection between the two sets.
Suppose that after reading the first $m{-}1$ letters of the word we have constructed a partition of $\{1,2,\dots,m\}$
with $i{+}1$ blocks (which we can assume are ordered by increasing smallest element),
where $i$ is the number of $R$'s read so far. If the $m$th letter is an $L_s$,
we add element $m{+}2$ to the $(s{+}1)$st block; if it is an $R$, we put $m{+}2$ in a separate new block.
This proves that $|P_n|=B_{n-1}$.
\end{proof}

\end{document}